\documentclass[letterpaper,12pt,draft]{article}
\usepackage[utf8]{inputenc}
\usepackage[=layout=letterpaper,includefoot,nohead]{geometry}
\usepackage{setspace}
\usepackage{pageslts}
\usepackage{amsmath}
\usepackage{amsfonts}
\usepackage{amssymb}
\usepackage{amsthm}
\usepackage{indentfirst}
\usepackage{pdflscape}
\usepackage{tocloft}
\usepackage{titlesec}
\usepackage[sorting=nyt]{biblatex}
\usepackage{hyperref}
\hypersetup{
    pdftitle={A Comparison of Pair and Triple Partition Relations},
    pdfauthor={Jonathon Eric Beers},
    final=true
}

\renewcommand{\contentsname}{Table of Contents}

\renewcommand{\listtablename}{List of Tables}

\defbibheading{bibliography}{\centering \MakeUppercase{\refname}}

\theoremstyle{plain} % Used for Theorems/Propositions/Corollaries/Lemmae/Facts

\newtheorem{thm}{Theorem}
\newtheorem{cor}{Corollary}

\theoremstyle{definition} % Used for Definitions/Questions/Conjectures/Exercises
\newtheorem{defn}{Definition}
\newtheorem{quest}{Question}

\begin{document}
% \pagenumbering{gobble}
\pagenumbering{roman}
\singlespacing

\newcommand{\candidate}{Jonathon Eric Beers}
\newcommand{\thesismajor}{Mathematics}
\newcommand{\thesistitle}{A Comparison of Pair and Triple Partition Relations}

\newenvironment{sig}[1]{
	\begin{tabular}{@{}p{3.4in}@{\hspace{1em}}p{2in}@{}}
	\hrulefill & \hrulefill \\
    {#1} & Date\\
}
{
	\end{tabular}
}

\newcommand{\signature}[2]{
	\begin{sig}{#1}
    {#2} & \\
	\end{sig}
}

\newcommand{\signaturemajor}[2]{
	\begin{sig}{#1}
    {#2} & \\
	Major Professor & \\
	\end{sig}
}

% Thesis approval page

\null
\vspace{1in}
\begin{center}
THESIS APPROVAL
\end{center}
\vspace{5mm}
\begin{tabbing}
Candidate:\hspace{2cm} \= \candidate\\
\\
Major:\> \thesismajor\\
\\
Thesis Title:\> \thesistitle\\
\end{tabbing}
Approval:\\[1.5cm]
\noindent
\signaturemajor{Thomas Leathrum}{Professor of Mathematics}
\\[2.0\baselineskip]
\signature{Jaedeok Kim}{Professor of Mathematics}
\\[3.0\baselineskip]
\signature{David Dempsey}{Professor of Mathematics}
\\[3.0\baselineskip]
\signature{Andrea Porter}{Director, Graduate Studies}
\thispagestyle{empty}
\newpage

% Title page

\null
\vspace{1in}
\begin{center}
\MakeUppercase{\thesistitle}\\
\vspace{2cm}
A Thesis Submitted to the\\
Graduate Faculty\\
of Jacksonville State University\\
in Partial Fulfillment of the\\
Requirements for the Degree of\\
Master of Science\\
with a Major\\
in \thesismajor\\
\vspace{2cm}
By\\
\MakeUppercase{\candidate}
\vfill
Jacksonville, Alabama\\
May 4, 2018
\end{center}
\thispagestyle{empty}
\newpage

% Copyright page

% \pagenumbering{roman}
\setcounter{page}{3}
\null
\vfill
\begin{center}
copyright 2018\\
% All Rights Reserved\\[2cm]
This thesis is licensed under a Creative Commons Attribution 4.0 International License.  To view a copy of this license, visit \url{http://creativecommons.org/licenses/by/4.0/}.\\[2cm]
\begin{sig}{\candidate}
\end{sig}
\end{center}
\newpage

% Abstract page

\doublespacing

\null
\vspace{1in}
\begin{center}
ABSTRACT\\
\end{center}
\hspace{5mm}This paper considers three different partition relations from partition calculus, two of which are pair relations and one of which is a triple relation.  An examination of the first partition relation and the ramification argument used to prove it will motivate questions regarding how to strengthen it.  These questions will lead to an examination of the second partition relation and its submodel argument where the answers to those questions will motivate further questions.  The final partition relation will then be compared to the prior two and an analysis of its strength will motivate final questions that will guide future mathematicians who wish to prove claims about positive or negative triple partition relations.
\par
% vii., 19 pages
\lastpageref{pagesLTS.roman}., \lastpageref{pagesLTS.arabic} pages
\newpage

% Acknowledgments page

\null
\vspace{1in}
\begin{center}
ACKNOWLEDGMENTS
\end{center}
\hspace{5mm}My eternal gratitude goes toward Dr.~Tom Leathrum for lifting me from na\"{i}ve set theory to ZFC and introducing me to partition calculus.%  I should include more to this.
\par
I would also like to thank Dr.~Jaedeok Kim and Dr.~David Dempsey for agreeing to sit on my committee and taking the time to familiarize themselves with partition calculus.%  I am indebted to them for their recommendation to include examples of simple partition relations in the introductory section.

% Auto-generated Table of Contents page

\newpage
\null
\vspace{1in}
\tableofcontents
\phantomsection
% \addcontentsline{toc}{section}{\MakeUppercase{\contentsname}}
\addcontentsline{toc}{section}{\texorpdfstring{\MakeUppercase{\contentsname}}{\contentsname}}

% Auto-generated List of Tables page

\newpage
\phantomsection
% \addcontentsline{toc}{section}{\MakeUppercase{\listtablename}}
\addcontentsline{toc}{section}{\texorpdfstring{\MakeUppercase{\listtablename}}{\listtablename}}
\null
\vspace{1in}
\listoftables

\vfill
\clearpage
\pagenumbering{arabic}
\setcounter{page}{1}

% These should be includes so each section is auto-newpaged.
\section[PREREQUISITES]{Prerequisites}
% \mysection{Prerequisites}
\subsection{A Note of Caution}
As with any subject, a fuller treatment of any results in partition calculus requires a great deal of background knowledge.  To discuss the three partition relations in depth we must assume some background knowledge in advanced set theory, but to make the results accessible we must review definitions not commonly covered in most graduate mathematics courses.  To compromise, only the bare minimum will be introduced without exploring their implications except insofar as an examination of our three partition relations do so.  Beyond that we must defer to various reference texts.

\subsection{Ordinals and Cardinals}
The objects of study are ultimately ordinals, which describe order isomorphisms between well-ordered sets, and cardinals, which describe injections between sets.

\begin{defn}
	Two totally ordered sets are \textbf{order isomorphic} iff there exists an order-preserving bijection between them.  Such sets are said to have the same \textbf{order type}.
\end{defn}

\begin{defn}[Ordinal]
	An \textbf{ordinal} is a set strictly well-ordered under element inclusion such that every element of the set is a subset of the set.  The \textbf{order type} of a well-ordered set $A$, denoted $\text{ot}(A)$, is the least ordinal $\alpha$ order isomorphic to $A$.
\end{defn}

A good reader can prove that $\varnothing$ is an ordinal, that $\alpha \cup \{\alpha\}$ is an ordinal for any ordinal $\alpha$, and a set of ordinals that contains every element of every ordinal it contains is itself an ordinal. By defining $0$ as $\varnothing$ and $\alpha + 1$, the ordinal successor of $\alpha$, as $\alpha \cup \{\alpha\}$ we can acquire any arbitrary finite ordinal.  By using Axiom of Infinity, which asserts the existence of a set that contains $\varnothing$ and contains the successor of $\alpha$ if it contains $\alpha$, we can acquire sets containing all finite ordinals.  Since ordinals are well-ordered there is a least ordinal that contains all finite ordinals.

\begin{defn}
	We denote by $\omega$ the least ordinal that contains all finite ordinals.
\end{defn}

\begin{defn}[Cardinal]
	Two sets $A$ and $B$ are \textbf{equinumerous} iff there exists a bijection between them.  The \textbf{cardinality} of a set $A$, denoted $|A|$, is the least ordinal equinumerous to $A$.  An ordinal $\kappa$ is a \textbf{cardinal} iff $|\kappa| = \kappa$.
\end{defn}

As a caveat, consider that if every set is equinumerous to some ordinal then every set can be well-ordered, but this is logically equivalent to the Axiom of Choice.  Since this paper intends to work in ZFC where the Axiom of Choice holds, our definition of cardinality remains well defined.
\par
Note that every finite ordinal is also a cardinal while no infinite successor ordinals are cardinals.  Since a set of cardinals less than some given ordinal is a subset of said given ordinal, cardinals are well-ordered and we can make sense of cardinal successors.  Since no infinite successor ordinal is a cardinal, the ordinal successor of an infinite cardinal is not the cardinal successor of an infinite cardinal.

\begin{defn}
	For any cardinal $\kappa$ we denote by $\kappa^{+}$ the least ordinal not equinumerous to $\kappa$ that contains $\kappa$.
\end{defn}

That it is not equinumerous means it must be a distinct cardinality.  That it contains $\kappa$ means that it does not precede $\kappa$.  That it is the least such ordinal means that it is its own cardinality and thus a cardinal.

\begin{defn}
	We denote by $\omega_{1}$ the least ordinal that contains all countable ordinals; equivalently, $\omega_{1}$ is the least uncountable ordinal.
\end{defn}

Since any countably infinite ordinal is equinumerous with $\omega$ and since $\omega_{1}$ is the least uncountable ordinal, thus $\omega^{+} = \omega_{1}$. 

\subsection{Partition Relations}

\begin{defn}
	Let $X$ be an ordered set, $Y$ a set, $\alpha$ an ordinal, and $\kappa$ a cardinal.  Then \[[X]^{\alpha} = \{A \subseteq X \mid \textrm{ot}(A) = \alpha\}\] and \[[Y]^{\kappa} = \{A \subseteq Y \mid |A| = \kappa\}.\]
%     \begin{align*}
%     	[X]^{\alpha} &= \{A \subseteq X \mid \textrm{ot}(A) = \alpha\} \\
%         [Y]^{\kappa} &= \{A \subseteq Y \mid |A| = \kappa\}
%     \end{align*}
\end{defn}

Equivalently, $[X]^{\alpha}$ is the set of all $\alpha$-order type subsets of $X$ and $[X]^{\kappa}$ is the set of all $\kappa$-cardinality subsets of $X$.  Since sets of order types 2 and 3 are also sets of cardinality 2 and 3 and since we limit our focus to pair and triple partition relations, the distinction between subsets of a given order type and subsets of a given cardinality will be unimportant.

\begin{defn}[Partitions]
	For $X$ a set and $I$ an indexing set, a \textbf{partition of $X$ in $I$ colors} is a function $\chi: X \rightarrow I$.  An \textbf{$r$-partition of $X$ in $I$ colors} is a partition of $[X]^{r}$ in $I$ colors.
\end{defn}

We represent the function with $\chi$, the first letter of the Greek word for ``color,'' since the function metaphorically paints every element of $[X]^{r}$ with some color.
\par
This modern definition corresponds to the one by Hajnal and Larson in \cite{hajnal+larson}.  An older, alternative definition in \cite{cst} that permits non-disjoint partitions defines a function $\pi$ so that $\pi(i) = \chi^{-1}(\{i\})$ and the union of the images under $\pi$ of elements from $I$ is $[X]^{r}$.

\begin{defn}[Partition Relations for Ordinals]
	Let $\alpha$ be an ordinal, $r$ a nonzero finite ordinal, $I$ any indexing set, and for all $i \in I$ let $\beta_{i}$ be an ordinal.  Then the partition relation
    \begin{align*}
    	\alpha \rightarrow \left(\beta_{i}\right)^{r}_{i \in I}
    \end{align*}
    pertains if and only if for all sets $A$ such that $\text{ot}(A) = \alpha$ and for all $r$-partitions $\chi$ of $A$ in $I$ colors there exists some $i \in I$ and some $B \subseteq A$ such that $\text{ot}(B) = \beta_{i}$ and $\chi([B]^{r}) \subseteq \{i\}$.  If for some $A$ and $\chi$ there are such $B$ and $i$, $B$ is said to be \textbf{homogeneous} in $A$ with respect to $\chi$ in color $i$.
\end{defn}

Concisely put, partition calculus is the study of finding the largest homogeneous subsets guaranteed to appear in any set of a given order type.
\par
To borrow terminology from Hajnal and Larson, one may informally refer to $\alpha$ as the \emph{resource}, to $I$ as the \emph{colors}, and to $\{\beta_{i}\}_{i \in I}$ as the \emph{goals} of the partition relation.  When $r = 2$, we refer to the partition relation as a pair partition relation; similarly, when $r = 3$ we refer to the partition relation as a triple partition relation.
\par
There are also partition relations for order types more generally and partition relations for cardinalities; however, we shall restrict attention to partition relations for ordinal resources, ordinal colors, and ordinal goals.
\par
In this paper we focus on $\omega_{1}$ as the resource and some ordinal as the colors.

% Move the cofinality definition to the Unbalanced Erdos-Rado section

\begin{defn}[Cofinality]
	Let $\alpha$ and $\beta$ be ordinals.  A function $f: \alpha \rightarrow \beta$ maps $\alpha$ \textbf{cofinally} to $\beta$ if for all $b \in \beta$ there exists some $a \in f(\alpha)$ such that $b \leq a$.  The \textbf{cofinality} of $\beta$, denoted $\text{cf}(\beta)$, is the least ordinal $\alpha$ for which there exists a function that maps $\alpha$ cofinally into $\beta$.  An ordinal $\alpha$ is \textbf{regular} iff $\text{cf}(\alpha) = \alpha$.
\end{defn}
\section[THE BALANCED ERD\H{O}S-RADO PARTITION RELATION]{The Balanced Erd\H{o}s-Rado Partition Relation}
% \mysection{The Balanced Erd\H{o}s-Rado Partition Relation}
\subsection{Outline of the Theorem}
We now introduce the first partition relation we wish to discuss and use to lay the foundation for questions regarding the other two partition relations.

% Theorem 2.9 in Hajnal-Larson with r = 2

\begin{thm}[Balanced Erd\H{o}s-Rado Partition Relation]
	\label{bpair}
	For all infinite cardinals $\kappa$ and for all nonempty ordinals $\gamma < \text{cf}(\kappa)$,
    \begin{align*}
    	\left(2^{<\kappa}\right)^{+} \rightarrow (\kappa + 1)^{2}_{\gamma}.
    \end{align*}
\end{thm}
\begin{proof}
	Refer to \cite{hajnal+larson} for the full proof.
\end{proof}

Although a full proof is beyond the scope of this work, we will provide a proof sketch to motivate our questions.  The proof of this theorem comes from the so-called Stepping-Up Theorem which uses a ramification argument by building a tree from transfinite sequences defined according to a coloring $\chi: (2^{<\kappa})^{+} \rightarrow \gamma$ and guaranteeing the tree has a branch of the required order type that is homogeneous with respect to the coloring.  In particular, it builds sequences of all order types less than $\kappa$ so that the elements after the first one form ``endhomogeneous'' sets with the first element.  When the tree is built, the tree up to the $\kappa$ level has cardinality no greater than $2^{\kappa}$ and every element at level $\kappa$ can be unioned with some branch of order type $\kappa$ to produce an ``endhomogenous'' set of order type $\kappa + 1$.  From this there must exist a homogeneous set of order type $\kappa + 1$.
\par
From the balanced Erd\H{o}s-Rado partition relation we may trivially deduce the following corollary.

\begin{cor}
	For all finite ordinals $n > 0$,
    \begin{align*}
    	\omega_{1} \rightarrow (\omega + 1)^{2}_{n}.
    \end{align*}
\end{cor}
\begin{proof}
	Consider that $\left(2^{<\omega}\right)^{+} = \left|\bigcup_{n \in \omega} 2^{n}\right|^{+}$, but a countable union of finite ordinals is $\omega$, $\omega$ is a cardinal, and $\omega^{+} = \omega_{1}$ since every countable ordinal can be placed into bijection with $\omega$ and will thus have the same cardinality as $\omega$.  Consider next that $\text{cf}(\omega) = \omega$, so if $\gamma \in \omega$ then $\gamma$ is a finite ordinal.
\end{proof}

We shall focus our attention on this form of the theorem.  To make comparisons with the unbalanced partition relations later in the paper more explicit we will rewrite the relation as $\omega_{1} \rightarrow (\omega + 1, (\omega + 1)_{n})^{2}$.
\par
Hajnal and Larson originally present the Erd\H{o}s-Rado theorem in even greater generality than Theorem \ref{bpair}.  By taking successive power sets of the cardinal successor of an infinite cardinal they can prove equivalent theorems for successive goals and successive exponents.  Consequently there is a triple relation analogue of Theorem \ref{bpair}.

\begin{cor}
	\label{btriple}
	For all finite ordinals $n > 0$,
    \begin{align*}
    	2^{\omega_{1}} \rightarrow (\omega + 2)^{3}_{n}.
    \end{align*}
\end{cor}

\subsection{Analysis of the Theorem}
Certain questions naturally arise about the goals in both the original general form and in our instance taken by setting $\kappa = \omega$.  From $\omega_{1}$ elements the ramification argument can construct a tree with branches that reach up to level $\omega$ in such a way that some branch is guaranteed to be homogeneous in any of finitely many colors.  Can we increase all the goals to $\omega + m$ for an arbitrary finite ordinal $m$?  As a matter of fact, by increasing the resource and the exponent we can increase every goal to $\omega + r - 1$.  Is there a way to do this without increasing the resource?  Can we ``breach the limit'' and increase the zeroth goal to $\omega + \omega$?  We can also consider what might happen if we release the restriction on having $b_{i} = b_{j}$ for all $i,j < n$; that is, if we ``unbalance'' the partition relation.  Can we increase the zeroth goal to $\omega + m$ for an arbitrary finite ordinal $m$ or to $\omega + \omega$?  Can we even ``reach the resource'' and increase the zeroth goal to $\omega_{1}$?
\section[THE UNBALANCED ERD\H{O}S-RADO PARTITION RELATION]{The Unbalanced Erd\H{o}s-Rado Partition Relation}
% \mysection{The Unbalanced Erd\H{o}s-Rado Partition Relation}
\subsection{Outline of the Theorem}
In response to some of the questions from the previous section, especially with regard to unbalancing the relation, we introduce the Unbalanced Erd\H{o}s-Rado Partition Relation.

\begin{thm}[Unbalanced Erd\H{o}s-Rado Partition Relation]
	For all infinite cardinals $\kappa$ and for all nonempty ordinals $\gamma < \text{cf}(\kappa)$,
    \begin{align*}
    	\left(2^{<\kappa}\right)^{+} \rightarrow \left(\left(2^{<\kappa}\right)^{+}, (\text{cf}(\kappa) + 1)_{\gamma}\right)^{2}.
    \end{align*}
\end{thm}
\begin{proof}
	Refer to \cite{hajnal+larson} for the full proof.
\end{proof}

For the sake of legibility throughout the section and keeping the notation from ``Partition Relations,'' let $\lambda = 2^{<\kappa}$.  The substituted partition relation then reads $\lambda^{+} \rightarrow (\lambda^{+}, (\text{cf}(\kappa) + 1)_{\gamma})^{2}$.

\subsubsection{Proof Sketch Preliminaries}
Before we can provide a proof sketch we will need to introduce some background from model theory, much of which can be found in Chapter IV of \cite{kunen}.  We assume the reader is familiar with structures.

\begin{defn}
	A structure is a \textbf{model of a theory} iff it satisfies every axiom of that theory.  In particular, a structure is a model of Zermelo-Frankel Set Theory with Axiom of Choice (ZFC) iff it satisfies every axiom of ZFC.
\end{defn}

\begin{defn}
	For any two structures $\mathcal{A}$ and $\mathcal{B}$, $\mathcal{A}$ is an \textbf{elementary substructure} of $\mathcal{B}$ iff $\mathcal{A}$ is a substructure of $\mathcal{B}$ and for any well formed formula $\phi(\vec{x})$ with $n$ free variables and for any $n$-tuple $\vec{a}$ of elements from the domain of $\mathcal{A}$, $\mathcal{A} \vDash_{\vec{x}/\vec{a}} \phi(\vec{x})$ iff $\mathcal{B} \vDash_{\vec{x}/\vec{a}} \phi(\vec{x})$.
\end{defn}

\begin{defn}
	For any set $A$, consider the following recursive definition:
    \begin{align*}
    	\bigcup^{0}A &= A \\
        \bigcup^{\alpha + 1}A &= \bigcup\bigcup^{\alpha}A && \text{$\alpha$ is an ordinal.}
    \end{align*}
    We define the \textbf{transitive closure} of A, denoted $\text{trcl}(A)$, as $\bigcup_{n < \omega}\left\{\bigcup^{n}A\right\}$.
\end{defn}

We provide a recursive definition, which only defines the base case at $0$ and the successor ordinal case, rather than a transfinitely recursive definition, which additionally defines the limit ordinal case, because the Axiom of Foundation guarantees that there are no infinite descending sequences of sets with respect to element inclusion.

\begin{defn}
	The \textbf{hereditary cardinality} of a set $A$ is the cardinality of its transitive closure.  For any infinite cardinal $\kappa$, let $H(\kappa) = \{x \mid |\text{trcl}(x)| < \kappa\}$.  We say that $H(\kappa)$ is the collection of all sets of hereditary cardinality less than $\kappa$.
\end{defn}

When $\kappa$ is a regular cardinal such that $\kappa > \omega$, $H(\kappa)$ is a model of ZFC without the Axiom of Power Set (henceforth, ZFC-P).  When furthermore the cardinality of the power set of every ordinal in $\kappa$ is less than $\kappa$, $H(\kappa)$ is a model of ZFC.  Kunen proves this in \cite{kunen}.  In particular, $H(\lambda^{++})$ is a model of ZFC-P.
\par
Aside from its significance as a model of ZFC-P, every set and every transitive closure of every set in $H(\lambda^{++})$ has cardinality no greater than $\lambda^{+}$, the resource of the unbalanced Erd\H{o}s-Rado partition relation. When they consider $H(\lambda^{++})$ to house their work, Hajnal and Larson consider a ``suitable'' elementary submodel $N$ such that, among other properties, its intersection with $\lambda^{+}$ is an ordinal.  They refer to this unique ordinal as the ``critical ordinal'' $\alpha$ of $N$.

\begin{defn}
	For any set $X$, a family $\mathcal{I} \subset \mathcal{P}(X)$ is an \textbf{ideal on $X$} iff the following properties hold:
    \begin{enumerate}
    \item $\varnothing \in \mathcal{I}$ and $X \notin \mathcal{I}$
    \item $\mathcal{I}$ is closed under subsets
    \item $\mathcal{I}$ is closed under finite unions
    \end{enumerate}
\end{defn}

\begin{defn}
	For any set $X$, a family $\mathcal{F} \subset \mathcal{P}(X)$ is a \textbf{filter on $X$} iff the following properties hold:
    \begin{enumerate}
    \item $\varnothing \notin \mathcal{F}$ and $X \in \mathcal{F}$
    \item $\mathcal{F}$ is closed under supersets
    \item $\mathcal{F}$ is closed under finite intersections
    \end{enumerate}
\end{defn}

In a sense, ideals of $X$ are a collection of ``small'' subsets of $X$ while filters of $X$ are a collection of ``big'' subsets of $X$.  The empty set is ``obviously'' small, the full set $X$ is ``obviously'' big, if a set is small then any subset of it must also be small, and if a set is big then any superset of it must also be big.  Ideals further impose that small sets are small enough that even finite unions won't be enough to create big sets and filters further impose that big sets are big enough that even finite intersections won't be enough to create small sets.

\begin{defn}
	For any limit ordinal $\lambda$, a subset $C \subseteq \lambda$ is \textbf{closed in $\lambda$} iff for all $\alpha \in \lambda$ if $\sup(C \cap \alpha) = \alpha$ then $\alpha \in C$.  A subset $C \subseteq \lambda$ is \textbf{unbounded in $\lambda$} iff for all $\alpha \in \lambda$ there exists some $\beta \in C$ such that $\alpha < \beta$.  A subset $C \subseteq \lambda$ is a \textbf{club in $\lambda$} iff $C$ is closed and unbounded in $\lambda$.
\end{defn}

\begin{defn}
	For any limit ordinal $\lambda$, a subset $S \subseteq \lambda$ is \textbf{stationary in $\lambda$} iff for all clubs $C$ in $\lambda$, $S \cap C \neq \varnothing$.
\end{defn}

For some critical ordinal $\alpha$ of some suitable elementary submodel $N \subseteq H(\lambda^{++})$, Hajnal and Larson construct an ideal $\mathcal{I}$ on $\alpha$ and a filter $\mathcal{F}$ on $\lambda^{+}$ such that every element of the filter is an element of $N$ stationary in $\lambda^{+}$ and for every element of the ideal the cardinality of its intersection with some element of the filter is less than $\kappa$.

\begin{defn}
	For $\alpha$ the critical ordinal of some suitable elementary submodel $N \subseteq H(\lambda^{++})$, some subset $X \subseteq \alpha$ \textbf{reflects the properties of $\alpha$} iff $X \cap F \neq \varnothing$ for all $F \in \mathcal{F}$.
\end{defn}

Since $X$ reflects the properties of $\alpha$ if $X \notin \mathcal{I}$, the ideal $\mathcal{I}$ can be referred to as the  non-reflecting ideal on $\alpha$.  Reflection will be important in the proof to find an element that can be added to some homogeneous set that will keep the set homogeneous.

\subsubsection{Proof Sketch}
To prove the unbalanced Erd\H{o}s-Rado partition relation Hajnal and Larson first consider some 2-coloring $\chi$ of $\lambda^{++}$ in $\gamma$ colors a suitable elementary submodel $N \subseteq H(\lambda^{++})$ such that $\chi \in N$.  From the submodel $N$ they consider the critical ordinal $\alpha$ and build the non-reflecting ideal $\mathcal{I}$.  For convenience, let us follow them in defining $\chi(\alpha; \eta) = \{\beta < \alpha \mid \chi(\{\alpha, \beta\}) = \eta\}$.
\par
If there exists some nonzero color $\eta \in \gamma$ such that $\chi(\alpha; \eta) \cap \alpha$ reflects the properties of $\alpha$, then there must be some homogeneous subset of that reflecting set in color $\eta$ and thus a maximal such subset $Z$.  If $|Z| < \text{cf}(\kappa)$, then $|Z| < \text{cf}(\alpha)$ by suitability of $N$.  This implies $Z$, let alone its intersection with $\alpha$, is not cofinal in $\alpha$, so $\sup(Z\cap \alpha) < \alpha$ and thus by suitability of $N$, $Z \cap \alpha \in N$.  The intersection $A = \bigcap\{\chi(z; \eta) \mid z \in Z \cap \alpha\}$, the set of all ordinals less than every element of $Z \cap \alpha$ whose pairing with any element of $Z \cap \alpha$ will have color $\eta$, will be an element of $N$ that contains $\alpha$.  Since it is in $N$ and $\alpha$ is in it, the set is an element of $\mathcal{F}$.  Since $\chi(\alpha; \eta) \cap \alpha$ reflects the properties of $\alpha$, so does $\chi(\alpha; \eta) \cap \alpha - \sup(Z \cap \alpha)$.  Since that reflects, its intersection with $A$ is nonempty and so there is an ordinal not already in $\sup(Z \cap \alpha)$ such that when paired with any element in $Z$ will have color $\eta$.  By adding that element to $Z$ a homogeneous subset greater than the maximal homogeneous subset can be constructed.  This contradicts the maximality of $Z$, so $|Z| \geq \text{cf}(\kappa)$.
\par
If there does not exist some nonzero color $\eta \in \gamma$ such that $\chi(\alpha; \eta) \cap \alpha$ reflects the properties of $\alpha$, then $\alpha - \chi(\alpha; 0)$, a subset of the union of $\gamma$ many elements of $\mathcal{I}$, must be in $\mathcal{I}$ since $\mathcal{I}$ is closed under unions of less than $\text{cf}(\kappa)$ many elements.  By the construction of $\mathcal{I}$, there is some $Z \in N$ such that $Z \subseteq \lambda^{+}$, $\alpha \in Z$, and $|Z - \chi(\alpha; 0)| < \kappa$.  The set $W = \{\beta \in Z \mid |Z - \chi(\beta; 0)| < \kappa\}$ is an element of $N$ and contains $\alpha$, so $W \in \mathcal{F}$, but since every element of the filter is stationary in $\lambda^{+}$, thus $W$ is stationary in $\lambda^{+}$.  It can then be proven that there is a stationary subset of $W$ that is homogeneous in color 0 with respect to $\chi$.
\par
From this we may trivially deduce the following corollary.

\begin{cor}
	For all finite ordinals $n > 0$,
    \begin{align*}
    	\omega_{1} \rightarrow (\omega_{1}, (\omega + 1)_{n})^{2}.
    \end{align*}
\end{cor}
\begin{proof}
	As in the proof for the balanced relation, $(2^{<\omega})^{+} = \omega_{1}$ and $\text{cf}(\omega) = \omega$.
\end{proof}

\subsection{Analysis of the Theorem}
In response to our questions regarding the balanced form, we can indeed guarantee an uncountable homogeneous subset of $\omega_{1}$ by unbalancing the partition relation.  By unbalancing the relation Hajnal and Larson can consider a case where some set reflects for a nonzero color, where the set in question consists of ordinals that all produce the same color when paired with the critical ordinal, or a case where no set reflects for any nonzero color.  In the case where some set reflects, they take advantage of reflection to prove some homogeneous subset of it must have cardinality no less than $\text{cf}(\kappa)$ and thus show that, when paired with the critical ordinal, has order type exactly $\text{cf}(\kappa) + 1$.  In the case where no such set reflects, they demonstrate the existence of a set containing the critical ordinal that is stationary in $\lambda^{+}$ and thus produce a homogeneous set of order type $\lambda^{+}$.

\section[ALBIN JONES'S TRIPLE PARTITION RELATION]{Albin Jones's Triple Partition Relation}
% \mysection{Albin Jones's Triple Partition Relation}
\subsection{Outline of the Theorem}
After looking at pair partition relations with resource $\omega_{1}$ we shall now introduce a triple partition relation with resource $\omega_{1}$ due to Albin Jones in \cite{jones}.

\begin{thm}
	For all $m,n < \omega$,
    \begin{align*}
    	\omega_{1} \rightarrow (\omega + m, n)^{3}.
    \end{align*}
\end{thm}
\begin{proof}
	Refer to \cite{jones} for full proof.
\end{proof}

Albin Jones splits the proof of the relation into three parts:  he proves the relation under the assumption that two cardinals are equal and strictly greater than $\omega_{1}$, then proves that there exists a forcing method to ``expand'' the universe to \emph{force} the assumption to hold, then proves that if it is possible to force the assumption to hold then the conclusion must have been true regardless of the forcing.  The second and third parts of the proof are outside the scope of his paper, merely called upon from others' works as needed, and so it is outside the scope of ours; we shall focus more attention to the first part of the proof and how the assumption he makes permits the partition relation to hold.

\subsubsection{Proof Sketch Preliminaries}

\begin{defn}
	A family of sets $\mathcal{A}$ is a \textbf{filter base} iff any intersection of finitely many of its elements is infinite.
\end{defn}

\begin{defn}
	A set $X$ is a \textbf{psuedo-intersection of a filter base $\mathcal{A} \subseteq [\omega]^{\omega}$} iff for all $A \in \mathcal{A}$, $|X - A| < \omega$.  The cardinal $\mathfrak{p}$ is the least cardinality of a filter base for which there does not exist a pseudo-intersection.
\end{defn}

\begin{defn}
	A filter $\mathcal{F}$ on $[\omega]^{\omega}$ is a \textbf{Ramsey filter} iff for all $m$-partitions $\chi$ of $\omega$ in $n$ colors for all finite ordinals $m$ and $n$ there exists some $F \in \mathcal{F}$ and some $i < n$ such that $\chi(F) = \{i\}$.
\end{defn}

\begin{defn}
	Suppose $x$ and $y$ are subsets of some ordinal.  We say $x << y$ iff every element of $x$ is less than every element of $y$.
\end{defn}

If $\mathfrak{p} = \mathfrak{c}$, where $\mathfrak{c} = |\mathbb{R}| = |2^{\omega}|$ is the cardinality of the reals, then there exist Ramsey filters.  If $\mathfrak{p} > \omega_{1}$ then every filter base $\mathcal{F} \subseteq [\omega]^{\omega}$ such that $|\mathcal{F}| \leq \omega_{1}$ has a pseudo-intersection.  These facts about $\mathfrak{p}$ are why the first part of the proof uses the assumption $\mathfrak{p} = \mathfrak{c} > \omega_{1}$.

\subsubsection{Proof Sketch}
Since $\omega_{1} \rightarrow (\omega + m, n)^{3}$ trivially pertains for $n \leq 3$ Jones fixes arbitrary $m < \omega$ and inducts on $n$.  Let $\chi$ be an arbitrary $3$-partition of $\omega_{1}$ in $2$ colors.  Since there exists an order isomorphism $\phi_{A}$ between $[\omega]^{k}$ under the lexicographic order and any $A \in [\omega_{1}]^{\omega^{k}}$, then for any $k < \omega$, $j < k$, $A \in [\omega_{1}]^{\omega^{k}}$, and $\beta \in \omega_{1}$ such that every element in $A$ is less than $\beta$, we can create a $(2k-j)$-partition $\chi_{k,j,A,\beta}$ of $\omega$ in $2$ colors that maps $x \in [\omega]^{2k-j}$ to $\chi(\{\phi_{A}(a \cup b), \phi_{A}(a \cup c), \beta\})$ given that $x = a \cup b \cup c$, $|a| = j$, $|b| = |c| = k - j$, and $a << b << c$.  This effectively assigns $x,y \in [\omega]^{k}$ to the color of $\{\phi_{A}(x), \phi_{A}(y), \beta\}$ under $\chi$ where the least $j$ elements of $x$ and $y$ are $x \cap y$ and $x \setminus y << y \setminus x$.  Since $\mathfrak{p} = \mathfrak{c}$ there exists a Ramsey filter $\mathcal{F}$ and some $F_{j,k,A,\beta} \in \mathcal{F}$ that is homogeneous with respect to $\chi_{j,k,A,\beta}$.  Consider that the intersection of the $F_{j,k,A,\beta}$ over all $j < k$, which we will denote as $F_{k,A,\beta}$, is a finite intersection of filter elements and thus is also a filter element.  Since $F_{k,A,\beta} \subseteq F_{j,k,A,\beta}$ for all $j < k$, thus it is homogeneous with respect to $\chi_{j,k,A,\beta}$ for all $j < k$.
\par
Jones defines two propositions with respect to $k$.

\begin{defn}[\textbf{$\Phi_{k}$}]
	There exist $j < k$, $A \in [\omega_{1}]^{\omega^{k}}$, and $B \in [\omega_{1}]^{\omega_{1}}$ such that $A << B$ and $\chi_{j,k,A,\beta}([F_{j,k,A,\beta}]^{2k-j}) = 0$ for all $\beta \in B$.
\end{defn}

\begin{defn}[\textbf{$\Psi_{k}$}]
	There exists $A \in [\omega_{1}]^{\omega^{k}}$ such that $\chi(\{\phi_{A}(x), \phi_{A}(y), \phi_{A}(z)\}) = 1$ for all  $x,y,z \in [\omega]^{k}$ such that $x \cap y << x \setminus y << y \setminus x$, $y \cap z << y \setminus z << z \setminus y$, and $|x \cap z| < |x \cap y|$.
\end{defn}

Jones first proves that if $\Phi_{k}$ pertains for some $k < \omega$ then $\omega_{1} \rightarrow (\omega + m, n + 1)^{3}$.  Since $\mathcal{F}$ is a Ramsey filter then for all $\beta, \gamma \in B$ there exists some $F_{k,A,\beta,\gamma} \in \mathcal{F}$ and some $i < 2$ such that $\chi(\{\phi_{A}(x),\beta,\gamma\}) = i$ for all $x \in [F_{k,A,\beta,\gamma}]^{k}$.  Let $f$ be a $2$-partition of $B$ in $2$ colors such that $f(\{\beta, \gamma\}) = i$ iff there exists $F_{k,A,\beta,\gamma} \in \mathcal{F}$ such that $\chi(\{\phi_{A}(x),\beta,\gamma\}) = i$ for all $x \in [F_{k,A,\beta,\gamma}]^{k}$.  Since $\omega_{1} \rightarrow (\omega_{1}, \omega)^{2}$, either there exists $B_{1} \in [B]^{\omega_{1}}$ homogeneous in $B$ with respect to $f$ in color $1$ or there exists $B_{0} \in [B]^{\omega}$ homogeneous in $B$ with respect to $f$ in color $0$.  In the color $1$ case, since $\omega_{1} \rightarrow (\omega + m, n)^{3}$ by the inductive hypothesis, either there exists $C_{1} \in [B_{1}]^{\omega + m}$ homogeneous with respect to $\chi$ in color $0$ or there exists $D_{1} \in [B_{1}]^{n}$ homogeneous with respect to $\chi$ in color $1$.  Let $x$ be an arbitrary element of the intersection of all the $F_{k,A,\beta,\gamma}$ such that $\{\beta,\gamma\} \in [D_{1}]^{2}$; then $\{x\} \cup D_{1} \in [\omega_{1}]^{n + 1}$ and it is homogeneous with respect to $\chi$ in color $1$.  In the color $0$ case, since there exists some $r < \omega$ such that $r \rightarrow (m, n + 1)^{3}$, we can consider such $r$ and let $C_{0} \in [B_{0}]^{r}$.  Let $F'$ be the intersection of all the $F_{j,k,A,\beta} \in \mathcal{F}$ such that $\beta \in C_{0}$, let $F''$ be the intersection of all the $F_{k,A,\beta,\gamma} \in \mathcal{F}$ such that $\beta,\gamma \in C_{0}$, and let $F = F' \cap F''$.  Let $a \in [F]^{j}$ and let $b_{i} \in [F]^{k - j}$ for all $i \in \omega$ so that $a << b_{0}$ and $b_{i} << b_{i + 1}$.  Let $A' = \{\phi_{A}(a \cup b_{i} \mid i < \omega\}$.  If there is no set of order type $n + 1$ homogeneous in $\omega_{1}$ with respect to $\chi$ in color $1$, then there is some $A'' \in [A']^{\omega}$ such that $\chi([A'']^{3}) = \{0\}$ and some $C'_{0} \in [C_{0}]^{m}$ such that $\chi([C'_{0}]^{3}) = \{0\}$.  Their union has order type $\omega + m$ and is homogeneous with respect to $\chi$ in color $0$.
\par
Jones then proves for all $k < \omega$ that if $\Psi_{k}$ pertains then there is a set of order type $k + 1$ homogeneous in $\omega_{1}$ with respect to $\chi$ in color $1$.  For $k = n$, if $\Psi_{n}$ pertains then $\omega_{1} \rightarrow (\omega + m, n + 1)^{3}$.  Let $x_{0,k}$ be the ordinal $k$ and let $x_{j,k}$ for $0 < j \leq k$ be the set obtained by substituting the greatest $j$ elements of the ordinal $k$ with the greatest $j$ elements of the ordinal $k + T(j)$ where $T(j) = \sum_{i \leq j}i$.  Trivially by the assumption of $\Psi_{k}$, $\chi(\{\phi_{A}(x_{j_{0},k}), \phi_{A}(x_{j_{1},k}), \phi_{A}(x_{j_{2},k})\}) = 1$ for $j_{0} < j_{1} < j_{2} \leq k$.
\par
Jones finally proves for all $k < \omega$ that either $\Phi_{j}$ pertains for some $j \leq k$ or $\Psi_{k}$ pertains.  For $k = n$, this suffices to prove $\omega_{1} \rightarrow (\omega + m, n + 1)^{3}$.  For $X \in [\omega_{1}]^{\omega_{1}}$, let $W_{l}(X)$ be the set of all $A \in [X]^{\omega^{l}}$ such that $\chi(\{\phi_{A}(x), \phi_{A}(y), \phi_{A}(z)\}) = 1$ for all $x, y, z \in [\omega]^{l}$ such that $x \cap y << x \setminus y << y \setminus x$, $y \cap z << y \setminus z << z \setminus y$, and $|x \cap z| < |x \cap y|$.  To facilitate the proof, Jones defines a proposition $\Psi_{l}'$ that implies $\Psi_{l}$ and then inducts on $l < k$ to demonstrate that $\Psi_{k}'$, and thus $\Psi_{k}$, pertains.

\begin{defn}[\textbf{$\Psi_{l}'$}]
	For all $X \in [\omega_{1}]^{\omega_{1}}$ there exist $A \in W_{l}(X)$ and $B \in [X]^{\omega_{1}}$ such that $A << B$ and $\chi(\{\phi_{A}(x), \phi_{A}(y), \beta\}) = 1$ for all $x, y \in [\omega]^{l}$ such that $x \cap y << x \setminus y << y \setminus x$ and for all $\beta \in B$.
\end{defn}

Since $[X]^{\omega^{0}} = [X]^{1}$ and $[\omega]^{0} = \{\varnothing\}$, $\Psi_{0}'$ trivially pertains.  Suppose there does not exist $j \leq k$ such that $\Phi_{j}$ pertains and suppose for all $l < k$ that $\Psi_{l}'$ pertains; then there exist $A_{0} \in W_{l}(X)$ and $B_{0} \in [X]^{\omega_{1}}$ that satisfy $\Psi_{l}'$.  For all $i < \omega$ let $A_{i + 1} \in W_{l}(B_{i})$ and $B_{i + 1} \in [B_{i}]^{\omega_{1}}$ such that they satisfy $\Psi_{l}'$.  Let $A = \bigcup_{i < \omega}A_{i}$; then $A \in W_{l + 1}(X)$.  Let $B$ be the set of all $\beta \in X$ such that every element of $A$ is less than $\beta$ and for all $j \leq l$ there exists some $F_{j,l + 1,A,\beta} \in \mathcal{F}$ such that $F_{j,l + 1,A,\beta}$ is homogeneous in color $1$ with respect to $\chi_{j,l + 1,A,\beta}$.  Since $\Phi_{l + 1}$ fails to pertain, there are co-countably many elements of $B$ and so $|B| = \omega_{1}$.  Since $F_{j,l + 1,A,\beta}$ is homogeneous in color $1$ with respect to $\chi_{j,l + 1,A,\beta}$ for all $\beta \in B$ and $j \leq l$, $F_{l + 1,A,\beta}$ is homogeneous in color $1$ with respect to $\chi_{j,l + 1,A,\beta}$ for all $\beta \in B$ and $j \leq l$.  Since $\{F_{l + 1,A,\beta} \mid \beta \in B\}$ is a filter base of cardinality $\omega_{1}$ and since $\mathfrak{p} > \omega_{1}$ by assumption, thus there exists a pseudo-intersection $F$ of that filter base.  Since $F$ differs from every $F_{l + 1,A,\beta}$ in only finitely many elements, thus for all $\beta \in B$ there exists some $N_{\beta} < \omega$ such that $F \setminus N_{\beta} \subseteq F_{l + 1,A,\beta}$.  Since $B$ is uncountable and $\omega_{1}$ is regular, there exists some $N < \omega$ such that $N_{\beta} = N$ for uncountably many $\beta$; let $B' \in [B]^{\omega_{1}}$ be the set of all $\beta$ such that $N_{\beta} = N$.  This implies $F \setminus N \subseteq F_{l + 1,A,\beta}$ for all $\beta \in B'$ and thus $F \setminus N$ is homogeneous in color $1$ with respect to $\chi_{j,l + 1,A,\beta}$ for all $\beta \in B'$ and for all $j \leq l$.  Finally, let $A' = \phi_{A}([F \setminus N]^{l + 1})$; then $A' \in W_{l + 1}(X)$, so $A'$ and $B'$ satisfy $\Psi_{l + 1}'$.  Since for all $l \leq k$, $\Psi_{l}'$ implies $\Psi_{l + 1}'$, thus $\Psi_{k}'$ and $\Psi_{k}$.

\subsection{Analysis of the Theorem}
Despite keeping the same resource available and permitting for an unbalanced relation, by switching from a pair relation to a triple relation we suddenly have smaller goals and thus a significantly weaker relation than the unbalanced Erd\H{o}s-Rado Theorem!  Rather than the $\omega + 1$ in finitely many colors by virtue of reflecting sets in a critical ordinal of an elementary submodel of $H(\lambda^{++})$ we only have an arbitrary finite ordinal for one other color.  Rather than the full $\omega_{1}$ in the zeroth color by virtue of a stationary subset of $\lambda^{+}$ we only have $\omega + m$ in the zeroth color for arbitrary finite ordinal $m$.
\section[FINAL REMARKS]{Final Remarks}
% \mysection{Final Remarks}
All three partition relations studied can be stated in terms of resource $\omega_{1}$ and all three only consider a finite collection of colors when starting with a resource of $\omega_{1}$.  The balanced and unbalanced Erd\H{o}s-Rado theorems are pair relations that make statements about $2$-partitions whereas Jones's partition relation is a triple relation that makes a statement about $3$-partitions.  Through a ramification argument the balanced Erd\H{o}s-Rado partition relation can reach a goal of $\omega + 1$ in any of its colors.  Through an elementary submodel argument the unbalanced Erd\H{o}s-Rado partition relation can reach a goal of $\omega_{1}$ in one color or $\omega + 1$ in any of the other colors.  Through a Ramsey filter argument together with forcing and reflection arguments Albin Jones's triple partition relation can reach a goal of $\omega + m$ for any finite $m$ in one color or any finite $n$ in another color.
\par
By considering $3$-partitions instead of $2$-partitions, the goals suddenly weaken and the colors suddenly reduce!  The triple partition relation is strong enough to breech $\omega + 1$ and reach a goal of $\omega + m$ for any finite $m$ in one of its colors, reaching farther than the balanced pair partition relation can in any of its colors, but it does so at the cost of only reaching a finite goal in its other color.  The triple partition relation compares less favorably with the unbalanced pair partition relation, which not only reaches the same goal of $\omega + 1$ in all its colors but even reaches the resource in one of its colors.  These comparisons, emphasized in Table \ref{t:comps}, raise natural questions.

\begin{quest}
	Does $\omega_{1} \rightarrow (\omega + m, (n)_{k})^{3}$ for all $m,n,k < \omega$?
\end{quest}

\begin{quest}
	Does $\omega_{1} \rightarrow (\omega + m, \omega + 1)^{3}$ for all $m < \omega$?
\end{quest}

\begin{quest}
	Does $\omega_{1} \rightarrow (\omega_{1}, (\omega + 1)_{n})^{3}$ for all $n < \omega$?
\end{quest}

Also noteworthy is that the balanced and unbalanced Erd\H{o}s-Rado partition relations make claims about the resource $(2^{<\kappa})^{+}$ for infinite cardinals $\kappa$ more generally while Jones's partition relation is restricted to a claim about the resource $\omega_{1}$.  Can his triple partition relation be ``lifted'' to more general claims about the same resources?

\begin{landscape}
    \begin{table}
    \centering
    {\renewcommand{\arraystretch}{1.25}
    \begin{tabular}{| l | c | c | c |}
    \hline
    \textbf{Partition Relation} & Balanced Erd\H{o}s-Rado & Unbalanced Erd\H{o}s-Rado & Jones Triple \\ \hline
    \textbf{Relation Kind} & Pair & Pair & Triple \\ \hline
    \textbf{Resource} & $\omega_{1}$ & $\omega_{1}$ & $\omega_{1}$ \\ \hline
    \textbf{Colors} & Any finite & Any finite & 2 \\ \hline
    \textbf{Goal in Color $0$} & $\omega + 1$ & $\omega_{1}$ & $\omega + m$ for all $m < \omega$ \\ \hline
    \textbf{Goal in Other Colors} & $\omega + 1$ & $\omega + 1$ & $n$ for all $n < \omega$ \\ \hline
    \textbf{Proof Method} & Ramification & Elementary submodel & Forcing \\ \hline
    \end{tabular}}
    \caption{Comparison of partition relations with resource $\omega_{1}$.}
    \label{t:comps}
    \end{table}

    \begin{table}
    \centering
    {\renewcommand{\arraystretch}{1.25}
    \begin{tabular}{| l | c | c | c |}
    \hline
    \textbf{Partition Relation} & Balanced Erd\H{o}s-Rado & Unbalanced Erd\H{o}s-Rado & Conjectural Triple \\ \hline
    \textbf{Relation Kind} & Pair & Pair & Triple \\ \hline
    \textbf{Resource} & $(2^{<\kappa})^{+}$ & $(2^{<\kappa})^{+}$ & $(2^{<\kappa})^{+}$ \\ \hline
    \textbf{Colors} & $<\text{cf}(\kappa)$ & $<\text{cf}(\kappa)$ & ? \\ \hline
    \textbf{Goal in Color $0$} & $\kappa + 1$ & $(2^{<\kappa})^{+}$ & ? \\ \hline
    \textbf{Goal in Other Colors} & $\kappa + 1$ & $\text{cf}(\kappa) + 1$ & ? \\ \hline
    \textbf{Proof Method} & Ramification & Elementary submodel & ? \\ \hline
    \end{tabular}}
    \caption{Comparison of partition relations with resource $(2^{<\kappa})^{+}$ for infinite cardinals $\kappa$.}
    \label{t:conjcomps}
    \end{table}
\end{landscape}

\newpage
\null
\vspace{1in}
% \bibliographystyle{plain}
% \bibliography{thesis.bib}
\printbibliography
\phantomsection
% \addcontentsline{toc}{section}{\MakeUppercase{\refname}}
\addcontentsline{toc}{section}{\texorpdfstring{\MakeUppercase{\refname}}{\refname}}
\nocite{roitman}

\end{document}